\documentclass[11pt,reqno]{amsart}
\usepackage{amsmath,amsfonts,amssymb,amscd,amsthm,amsbsy,epsf}
\newtheorem{thm}{Theorem}

\newtheorem*{bliss}{Bliss Lemma} 
\theoremstyle{definition}
\newtheorem*{rems}{Remarks}

\def\mysection#1{\par\medskip\noindent {\bf #1.}\par}
\def\S{{\mathcal S}}
\def\real{{\mathbb R}}
\begin{document}

\title{On Hardy-Sobolev embedding}
\author{William Beckner}
\address{Department of Mathematics, The University of Texas at Austin,
1 University Station C1200, Austin TX 78712-0257 USA}
\email{beckner@math.utexas.edu}
\begin{abstract}
Linear interpolation inequalities that combine Hardy's inequality with 
sharp Sobolev embedding are obtained using classical arguments  of 
Hardy and Littlewood (Bliss lemma). 
Such results are equivalent to Caffarelli-Kohn-Nirenberg inequalities 
with sharp constants. 
A one-dimensional convolution inequality for the exponential density is
derived as an application of these methods. 
\end{abstract}

\maketitle

\mysection{1. Interpolation inequalities}

A classical problem in analysis is to understand how ``smoothness'' controls 
norms that measure the ``size'' of functions.
Maz'ya recognized in his classic text on Sobolev spaces the intrinsic 
importance of inequalities that would refine both Hardy's inequality and 
Sobolev embedding. 
Dilation invariance and group symmetry play an essential role in 
determining sharp constants. 
Recent interest has focused on how to add ``error terms'' to the 
classical estimates. 
The objective here is the following theorem drawn as a novel consequence 
of this effort to extend Sobolev embedding.  

\begin{thm}\label{musina}
For $f\in \S(\real^n)$, $n\ge 3$ and $2 < q \le q_* = \frac{2n}{n-2}$
\begin{equation}\label{eq:musina1}
\int_{\real^n} |\nabla f|^2\,dx 
\ge \frac{(n-1)(n-3)}4 \int_{\real^n} \frac1{|x|^2} |f|^2\,dx 
+ C_q \bigg[ \int_{\real^n} |x|^{-n(q_*-q)/q_*} |f|^q\,dx\bigg]^{2/q}
\end{equation}

\begin{equation*}
C_q = \left[ \frac{2\pi^{n/2}}{\Gamma (n/2)}\right]^{1/\alpha} 
\left(\frac{q}2\right)^{q/2}
\left[\frac{\Gamma (\alpha)\Gamma(\alpha+1)}
{\Gamma (2\alpha)}\right]^{1/\alpha} \ ,\qquad 
\alpha = \frac{q}{q-2}\ .
\end{equation*}
\end{thm}

It is important to emphasize that this is a global estimate on $\real^n$, 
and that it improves over a convex linear combination of the two terms
on the right-hand side with their respective sharp constants. 
A comparison can be made with Hebey's $AB$ program (see \cite{Hebey}, chapter~7)
where for this setting, optimal pairs of constants $A,B$ would be 
determined for 
\begin{equation*}
\int_{\real^n} |\nabla f|^2\,dx 
\ge A \int_{\real^n} \frac1{|x|^2} |f|^2\,dx 
+ B\bigg[ \int_{\real^n} |f|^{q^*} \,dx\bigg]^{2/q^*}
\end{equation*}
and additionally the treatment of forms for elliptic differential 
operators on hyperbolic space (see \cite1). 
This framework suggests that it would be interesting to study this 
inequality for values of $A$ less than the sharp value of $(n-2)^2/4$.
The sharp value of $B$ (with $A=0$) is the Sobolev embedding constant 
$\pi n(n-2) [\Gamma (n/2)/\Gamma (n)]^{2/n}$.
Moreover, as this is an  {\em a~priori} inequality, the limiting extremal 
functions may be singular at the origin. 
A further unexpected feature is that such inequalities are equivalent 
to three-dimensional Sobolev embedding estimates. 

\begin{proof}[Proof of Theorem~\ref{musina}] 
Under radial decreasing rearrangement, the gradient term decreases and 
the terms on the right-hand side increase so it suffices to prove this 
inequality for radial decreasing functions. 
By setting $f(x) = u(x)|x|^{-(n-1)/2}$ with $u(0)=0$, then 
inequality \eqref{eq:musina1} is equivalent to 
\begin{equation}\label{musina-equiv}
\int_0^\infty \Big| \frac{\partial u}{\partial r}\Big|^2 \,dr 
\ge (\omega_{n-1})^{2/q\  -\  1}\ C_q 
\bigg[ \int_0^\infty |u|^q r^{-\frac{q}2 -1}\, dr\bigg]^{2/q}
\end{equation}
where $\omega_{n-1} = $ surface area of the unit sphere $S^{n-1}$. 
Since $|f(x)| \le \sigma (x)|x|^{-\frac{n}2 +1}$ with 
$\lim_{|x|\to0} \sigma (x)$ $ =0$, observe that $|u(x)|^2 |x|^{-1} \le 
\sigma^2 (x) \to0$ as $|x|\to 0$, and there is no difficulty with 
the accompanying integration by parts in the functional transformation 
of the inequality. 
Further, for $w= |x|^{-1} u$
\begin{equation}\label{eq3}
\int_0^\infty \left(\frac{\partial u}{\partial r}\right)^2\,dr 
= \int_0^\infty r^2 \left(\frac{\partial w}{\partial r}\right)^2\,dr 
\ge (\omega_{n-1})^{2/q\ -\ 1} \ C_q 
\bigg[ \int_0^\infty |w|^q\ r^{\frac{q}2 -1} \,dr\bigg]^{2/q} \ .
\end{equation}
Then set $h(r) = \partial w/\partial r$ so that 
$w(r) = - \int_r^\infty h(t)\,dt$, and 
\begin{equation*}
\int_0^\infty r^2 |h|^2\,dr 
\ge (\omega_{n-1})^{2/q\ -\ 1}\ C_q \bigg[ \int_0^\infty \Big|\int_r^\infty 
h(t)\,dt\Big|^q\ r^{\frac{q}2-1}\,dr\bigg]^{2/q}\ .
\end{equation*}
Now set $g(r) = r^{-2} h(1/r)$, and change variables $r\to 1/r$
\begin{equation*}
\int_0^\infty |g|^2\,dr \ge (\omega_{n-1})^{2/q\ -\ 1} \ C_q 
\bigg[ \int_0^\infty \Big| \int_0^r |g(s)|\,ds \Big|^q 
r^{-\frac{q}2 -1}\,dr\bigg]^{2/q}\ .
\end{equation*}
Now to calculate $C_q$, apply the Bliss lemma (\cite{bliss}): 
\renewcommand{\qed}{}
\end{proof}

\begin{bliss} 
For $s\ge 0$, $q>p>1$, $r= \frac{q}p -1$ 
\begin{gather} 
\bigg[ \int_0^\infty \Big| \int_0^s g(t)\,dt \Big|^q s^{r-q}\,ds\bigg]^{p/q}
\le K\int_0^\infty |g|^p\,ds \label{eq:bliss}\\
\noalign{\vskip6pt}
K = (q-r-1)^{-p/q} \bigg[ \frac{r\Gamma (q/r)}{\Gamma (1/r)\Gamma ((q-1)/r)}
\bigg]^{rp/q}\ .\notag
\end{gather}
Equality is attained for functions of the form
$$g(s) = A(cs^r +1)^{-(r+1)/r}\ ,\qquad c>0\ .$$
\end{bliss}

Then $C_q$ in equation~\eqref{eq:musina1} is given by 
$$C_q = (\omega_{n-1})^{1/\alpha} \Big(\frac{q}2\Big)^{2/q} 
\bigg[\frac{\Gamma (\alpha)\Gamma (\alpha+1)}{\Gamma(2\alpha)}\bigg]^{1/\alpha}
$$
for $\alpha = q/(q-2)$ and $\omega_{n-1} = 2\pi^{n/2} /\Gamma (n/2)$. 
Tracing back the  functional transformations, an extremal function for 
inequality~\eqref{eq:musina1} is given by 
\begin{equation*}
f(x) = |x|^{-(n-3)/2} (1+|x|^\beta )^{-1/\beta}\ ,\qquad 
\beta = \frac{q}2 -1\ .
\end{equation*}
%
%
Perhaps conceptually it is surprising that since extremals do not exist 
for Hardy's inequality, still this linear combination with a Sobolev 
embedding term suffices to determine an extremal function 
where equality is attained. 
The connection with Caffarelli-Kohn-Nirenberg inequalities becomes 
more explicit with the following extension of Theorem~\ref{musina}. 

\begin{thm}\label{CKN}
For $f\in \S(\real^n)$, $n\ge 3$, $2< q\le q_* = \frac{2n}{n-2}$ and 
$0 < a < (n-2)/2$
\begin{gather}
\int_{\real^n} |\nabla f|^2\,dx 
\ge a(n-2-a) \int_{\real^n} \frac1{|x|^2} |f|^2\,dx + D_{q,a}
\bigg[ \int_{\real^n} |x|^{-n(q_* -q)/q_*} |f|^q\,dx \bigg]^{2/q}
\label{eq:CKN}\\ 
\noalign{\vskip6pt}
D_{q,a} = (n-2-2a)^{2/q\ +\ 1} C_q\notag
\end{gather}
\end{thm}

\begin{proof} 
Under radial decreasing rearrangement, the gradient term decreases and the 
terms on the right-hand side increase so it suffices to prove this 
inequality for radial decreasing functions. 
By setting $f(x) = u(x) |x|^{-a}$ with $u(0)=0$ and $0<a< (n-2)/2$, 
then inequality~\eqref{eq:bliss} is equivalent to 
\begin{equation}\label{eq:CKN-pf}
\int_0^\infty r^{n-1-2a} \Big| \frac{\partial u}{\partial r}\Big|^2\,dr 
\ge (\omega_{n-1})^{2/q\ - \ 1} D_{q,a} 
\biggl[\int_0^\infty |u|^q r^{\frac{q}2 [n-2a-2]-1}\, dr\bigg]^{2/q}\ .
\end{equation}
Since $|f(x)| \le \sigma (x) |x|^{-n/2\ + \ 1}$ with $\lim_{|x|\to0}
\sigma (x) =0$, observe that $|u(x)|^2 |x|^{n-2-2a} \le \sigma^2(x)\to0$
as $|x|\to 0$, and there is no difficulty with the accompanying integration 
by parts in the functional transformation of the inequality.
Note that for values of $a(n-2-a)$ less than the maximum 
$(n-2)^2 /4$, there are two roots and the restriction $0 < a < (n-2)/2$ means 
that the smaller root is selected. 
This contrasts with the method used in the proof of Theorem~\ref{musina}. 
Now make the change of variables $s= r^{n-2-2a}$; then 
\begin{equation}\label{eq:CKN-proof2}
\int_0^\infty s^2 \Big(\frac{\partial u}{\partial s}\Big)^2\,ds 
\ge (n-2-2a)^{-2/q\ -\ 1} (\omega_{n-1})^{2/q\ -\ 1} D_{q,a} 
\bigg[ \int_0^\infty |u|^q s^{q/2\ -\ 1}\,ds\bigg]^{2/q}
\end{equation}
and using the previous calculation from equation~\eqref{eq3} 
\begin{equation*}
D_{q,a} = (n-2-2a)^{(2/q\ +\ 1)} C_q
\end{equation*}
Surprisingly the dependence on the parameter $a$ is simple and allows 
an immediate recovery of Hardy's inequality by setting $a= (n-2)/2$. 
Further, setting $a= (n-3)/2$ which corresponds to the special case 
of Theorem~\ref{musina} gives $D_{q,a} = C_q$. 
Again as in the proof of Theorem~\ref{musina}, an extremal function 
for inequality~\eqref{eq:CKN} is given by 
\begin{equation*}
f(x) = |x|^{-a} (1+ |x|^{\beta (n-2-2a)})^{-1/\beta}\ ,\qquad 
\beta = \frac{q}2 -1\ .
\end{equation*}
And a close look at inequality~\eqref{eq:CKN-proof2} 
shows an equivalent three-dimensional inequality by using radial decreasing
symmetrization: 
\begin{gather}
\int_{\real^3} |\nabla u|^2\,dx 
\ge E_q \bigg[ \int_{\real^3} |u|^q |x|^{q/2\ -\ 3}\,dx\bigg]^{2/q} 
\label{eq:3dimen}\\
\noalign{\vskip6pt}
E_q =  (4\pi)^{1\ -\ 2/q} \Big(\frac{q}2\Big)^{2/q} 
\left[\frac{\Gamma(\alpha)\Gamma (\alpha+1)}
{\Gamma (2\alpha)}\right]^{1/\alpha}\ , \qquad 
\alpha = \frac{q}{q-2}\notag
\end{gather}
Symmetrization can be applied here since 
$$\frac{q}2 -3 \le \frac{q_*}2 -3 = \frac{n}{n-2} -3 \le 0\ \text{ for }\ 
n\ge 3\ .$$
This is simply the $a=0$, $n=3$ case of Theorem~\ref{CKN} so in fact 
the full framework of that theorem can be bootstrapped from this one case!
This inequality was first obtained by Glaser, Martin, Grosse and 
Thirring \cite{GMGT}, and then extended to the $n$-dimensional setting 
by Lieb \cite{Lieb}. 
\end{proof}

\begin{rems}
Several recent papers provide context for the development of the 
main theorems here.
Musina's short paper \cite{Musina} shows the connection between error estimates
for Hardy's inequality and the Caffarelli-Kohn-Nirenberg inequalities. 
{From} the nature of the proof of his Proposition~1.4 which holds only 
for radial functions on the unit ball in $\real^n$, one can obtain an 
equivalence with the global estimates in Theorem~\ref{musina}. 
Various forms of embedding estimates for fractional smoothness are given 
in the author's papers \cite{B-pams08}, \cite{2}. 
With respect to reduction to one-dimensional estimates, see especially 
section~5 in \cite{2}. 
After formulating the substance of this paper, the author became aware 
of the article by Dolbeault, Esteban, Loss and Tarantello \cite{DELT} 
which gives new results on existence and symmetry of extremals for the 
Caffarelli-Kohn-Nirenberg inequalities.
\end{rems}

\bigskip
\mysection{2. Applications --- Young's inequality} 

The weighted Hardy inequality~\eqref{eq:bliss} has varied applications, 
including the calculation of the sharp Sobolev embedding constant. 
In general, inequalities that occur for the ``line of duality'', 
$L^p\to L^{p'}$ with $1< p<2$ and $1/p + 1/p' =1$, are especially 
interesting; examples include the Hausdorff-Young theorem and the 
Hardy-Littlewood-Sobolev inequality. 
Here the Bliss lemma for $q=2$ is applied to calculate a one-dimensional 
convolution inequality on the line of duality. 

\begin{thm}[Young's inequality for an exponential density] 
For $f\in L^p(\real)$, $1<p<2$, $1/p + 1/p' =1$ and $\varphi (x)=e^{-|x|}$
\begin{gather}
\|\varphi * f\|_{L^{p'}(\real)} \le A_p \|f\|_{L^p(\real)} 
\label{eq:young}\\
\noalign{\vskip6pt}
A_p = \Big( \frac{p'}2\Big)^{2/p}
\left[ \frac{\Gamma (\frac{2p}{2-p})}
{\Gamma (\frac2{2-p}) \Gamma (\frac{p}{2-p})} 
\right]^{2/p\ -\ 1}\notag
\end{gather}
\end{thm}

\begin{proof} 
{From} the Bliss lemma for $q=2$:
\begin{equation*}
\int_0^\infty \Big| \int_0^s g(t)\,dt\Big|^2 s^{-1\ -\ 2/p'}\,ds 
\le K^{2/p} \bigg[ \int_0^\infty |g|^p\,ds\bigg]^{2/p}
\end{equation*}
First, set $g(t) = h(t) t^{-1/p}$, and then set $t= e^x$
\begin{gather*}
\int_0^\infty \Big| \int_0^s h(t) (t/s)^{1/p'} \frac1t\,dt\Big|^2 
\frac1s\,ds 
\le K^{2/p} \bigg[ \int_0^\infty |h|^p \frac1s \,ds\bigg]^{2/p}\\
\noalign{\vskip6pt}
\int_{-\infty}^\infty \Big| \int_{-\infty}^x h(y) e^{-(x-y)/p'}\,dy\Big|^2\,dx
\le K^{2/p} \bigg[ \int_{-\infty}^\infty |h|^p\,dx \bigg]^{2/p}\\
\noalign{\vskip6pt}
\Big| \int_{\real\times\real} h(u) e^{-2/p'|u-v|} h(v)\,du\,dv\Big| 
\le \frac2{p'} K^{2/p} \bigg[\int_{\real} |h|^p\,dx\bigg]^{2/p}\\
\noalign{\vskip6pt}
A_p = \Big(\frac{p'}2\Big)^{2/p\ -\ 1} K^{2/p} 
= \Big( \frac{p'}2\Big)^{2/p} 
\bigg[\frac{\Gamma (\frac{2p}{2-p})} 
{\Gamma(\frac{2}{2-p}) \Gamma (\frac{p}{2-p})}\bigg]^{2/p\ -\ 1}
\end{gather*}
An extremal function for inequality~\eqref{eq:young} is given by 
$$f(x) = \text{cosh} [p'x/(4p\delta)]^{-\delta}\ ,\qquad 
\delta = \frac2{2-p}\ .$$
\end{proof}

\bigskip
\mysection{3. Stein-Weiss potentials}

The pure  Sobolev embedding portion of Theorem~\ref{CKN} 
(with $a=0$) gives 
\begin{equation}\label{eq:pure}
\int_{\real^n} |\nabla f|^2\,dx 
\ge (n-2)^{2/q\ +\ 1} C_q 
\biggl[ \int_{\real^n} |x|^{-q\alpha} |f|^q \,dx\bigg]^{2/q}\ ,
\qquad \alpha = \frac{n}q - \frac{n}{q_*}
\end{equation}
which determines a Stein-Weiss potential map on the line of duality 
through application of the fundamental solution for the Laplacian: 

\begin{thm}\label{thm:fundamental}
For $h\in L^p (\real^n)$, $n>2$, $\frac{2n}{n+2} <p<2$, 
$(2n/p') - 2\alpha = n-2$
\begin{gather}
\Big| \int_{\real^n\times\real^n} h(x)|x|^{-\alpha} |x-y|^{-(n-2)} 
|y|^{-\alpha} h(y)\,dx\,dy\Big| 
\le A_\alpha (\|h\|_{L^p (\real^n)})^2
\label{eq:fundamental}\\
A_p = \Gamma (n/2) \left[ \frac{2\pi^{n/2}}{(n-2)\Gamma (n/2)}\right]^{2/q} 
\left( \frac2q\right)^{q/2} 
\left[\frac{\Gamma (2\delta)}
{\Gamma (\delta)\Gamma (\delta+1)}\right]^{1\ -\ 2/q}\notag
\end{gather}
with $q= p'$ and $\delta = p/(2-p)$.
\end{thm}

Observe  that $0<\alpha <1$. 
By applying symmetrization and using the dilation invariance, study of 
this inequality can be reduced to radial functions with an inversion 
symmetry: 
$h(|x|) = |x|^{-2n/p} h(1/|x|)$. 
Lieb \cite{Lieb} shows that extremal functions exist for general 
Stein-Weiss maps from $L^p (\real^n)$ to $L^q (\real^n)$ where $p\ne q$. 
In general, one would like to calculate sharp constants for this inequality 
in the ``line of duality'': 
\begin{equation}\label{eq:fundamental-pf}
\Big| \int_{\real^n\times\real^n} h(x) |x|^{-\alpha} |x-y|^{-\lambda} 
|y|^{-\alpha} h(y)\,dx\,dy \Big| 
\le A_{\lambda,\alpha} (\|h\|_{L^p(\real^n)})^2
\end{equation}
$1<p<2$, $\lambda  = (2n/p') - 2\alpha$ and $0<\alpha < n/p'$.
The case $\alpha =0$ corresponds to the Hardy-Littlewood-Sobolev 
inequality, and the case $\lambda = n-2$ corresponds to 
Theorem~\ref{thm:fundamental}. 
Sharp constants for such maps from $L^p(\real^n)$ to $L^p(\real^n)$ 
are calculated in \cite{B-pams08}. 

\bigskip
\mysection{4. Caffarelli-Kohn-Nirenberg inequalities}

One form of the Caffarelli-Kohn-Nirenberg inequalities is given by 
\begin{equation}\label{eq:CKN-inequal}
\int_{\real^n} |x|^{-2a} |\nabla u|^2 \,dx 
\ge D_{q,a} \bigg[ \int_{\real^n} |u|^q |x|^{-q[\frac{n}q - \frac{n}{q^*} +q]}
dx\bigg]^{2/q}
\end{equation}
where $n\ge 3$, $2< q< q_* = 2n/(n-2)$ and $a< \frac{n-2}2$.
Observe that for radial functions this is exactly 
inequality~\eqref{eq:CKN-pf} above.
In contrast to Theorem~\ref{CKN}, there is no reason here in the 
application of the Bliss lemma  why the parameter $a$ can not assume 
negative values. 
Hence the computation follows the same argument as above with 
$$D_{q,a} = (n-2-2a)^{(2/q +1)} C_q$$
and the corresponding radial extremal 
$$u(|x|) = (1+ |x|^{\beta (n-2-2a)})^{-1/\beta}\ ,\qquad 
\beta = \frac{q}2 -1\ .$$
More generally, this result extends from $L^2 (\real^n)$ to $L^p(\real^n)$ 
in the case of radial functions. 

\begin{thm}\label{thm:radialfunc}
For $u\in \S(\real^n)$, $1< p < q < \infty$, $n>p$ and 
$0 < a < (n-p)/p$ 
\begin{equation}\label{eq:radialfunc} 
\int_{\real^n} |x|^{-pa} |\nabla u|^p \,dx
\ge D_{p,q,a} \bigg[\int_{\real^n}|u|^q |x|^{-q[\frac{n}q - \frac{n}{q_*} +a]}
dx\bigg]^{p/q}
\end{equation}
with $u$ begin a radial function and $q^* = pn/(n-p)$. 
If $a=0$, then in addition the result holds for non-radial functions. 
\begin{equation*}	
\begin{split}
D_{p,q,a} &= \left[ \frac{n-p(a+1)}{p-1}\right]^{+(p-1)\ +\ p/q} 
\mkern-13mu
(\omega_{n-1})^{1\ -\ p/q} (q\ -\ q/p)^{-p/q}\\
\noalign{\vskip6pt}
&\qquad 
\left[ \frac{\Gamma [qp/(q-p)]}{\Gamma (q/(q-p))\Gamma ((q-1)p/(q-p))}\right]
^{1\ -\ p/q}\ .
\end{split}
\end{equation*}
\end{thm}

\begin{proof} 
For radial functions, this inequality can be written as 
\begin{equation*} 
\int_0^\infty r^{n-pa-1} \Big|\frac{\partial u}{\partial r}\Big|^p\, dr
\ge (\omega_{n-1})^{p/q\ -\ 1} D_{p,q,a} 
\bigg[\int_0^\infty r^{q(\frac{n}{q_*} -a)-1} |u|^q\,dr\bigg]^{p/q}\ .
\end{equation*}
Now make the change of variables $s= r^{(n-p(a+1))/(p-1)}$; then 
\begin{equation*}
\int_0^\infty s^{2p-1}\left(\frac{\partial u}{\partial s}\right)^p\,ds 
\ge \left[ \frac{n-p(a+1)}{p-1}\right]^{-(p-1)-\ p/q} 
\mkern-13mu
(\omega_{n-1})^{p/q\ -\ 1} D_{p,q,a} 
\bigg(\int_0^\infty \mkern-12mu 
|u|^q s^{\frac{q}p (p-1)-1} \,ds\bigg)^{p/q}\ .
\end{equation*}
Set $h=\frac{\partial u}{\partial s}$ so that 
$u(s) = - \int_s^\infty h(t)\,dt$, and then let 
$g(r) = r^{-2} h(1/r)$, and change variables $r\to 1/r$
\begin{equation*}
\int_0^\infty \mkern-12mu  |g|^p\,dr 
\ge \left[\frac{n-p(a+1)}{p-1}\right]^{-(p-1)-\ p/q} 
\mkern-13mu
(\omega_{n-1})^{p/q\ -\ 1} D_{p,q,a} 
\bigg[ \int_0^\infty \Big|\int_0^r \mkern-6mu |g(s)|\,ds\Big|^q 
r^{-\frac{q}p (p-1)-1}\,dr\bigg]^{p/q} \ .
\end{equation*}
Applying the Bliss lemma, one finds 
\begin{gather*}
\left[\frac{n-p(a+1)}{p-1}\right]^{-(p-1)-\ p/q} 
(\omega_{n-1})^{p/q\ -\ 1} D_{p,q,a} \\
\noalign{\vskip6pt}
= \Big(q-\frac{q}p\Big)^{-p/q} 
\bigg[ \frac{\Gamma (qp/(q-p)]}
{\Gamma(q/(q-p))\Gamma((q-1)p/(q-p))}\bigg]^{1\ -\ p/q}\ .
\end{gather*}
An extremal function for inequality~\eqref{eq:radialfunc} is given by 
$$u(|x|) = \left[ 1 + |x|^{\beta (n-p(a+1))/(p-1)} \right]^{-1/\beta}\ ,
\qquad \beta = \frac{q}p -1\ .$$
In the case $a=0$, apply radial decreasing symmetrization for reduction 
to radial functions. 
\end{proof}

\section*{Acknowledgement}

I would like to thank Jean Dolbeault for a helpful discussion on the 
existence of extremals. 


\end{document}